\documentclass[11pt, a4paper]{article}
\usepackage{amsmath, amssymb, amsthm, paralist}

\newtheorem{theorem}{Theorem}[section]
\newtheorem{lemma}[theorem]{Lemma}
\newtheorem{proposition}[theorem]{Proposition}

{\theoremstyle{definition}
\newtheorem{remark}[theorem]{Remark}

\newtheorem{definition}[theorem]{Definition}

}

\newcommand{\F}{\mathsf{F}}

\newcommand{\PSL}{\mathsf{PSL}}

\newcommand{\oct}{\mathcal{O}}

\newcommand{\mouf}{\mathbb{M}}

\newcommand{\PO}{{\mathbb{P}}^2(\oct)}
\newcommand{\hatPO}{\hat{\mathbb{P}}^2(\oct)}
\newcommand{\hatPOk}{\hat{\mathbb{P}}^2(\oct_k)}
\newcommand{\HO}{\mathcal{H}(\oct_3)}
\newcommand{\HOk}{\mathcal{H}((\oct_k)_3)}

\DeclareMathOperator{\ch}{char}
\DeclareMathOperator{\tr}{T}
\DeclareMathOperator{\T}{T}

\DeclareMathOperator{\N}{N}

\DeclareMathOperator{\Fix}{Fix}
\DeclareMathOperator{\id}{id}

\DeclareMathOperator{\Aut}{Aut}
\DeclareMathOperator{\Sym}{Sym}
\DeclareMathOperator{\Cent}{Cent}

\numberwithin{equation}{section}

\begin{document}

\title{Moufang sets arising from polarities of Moufang planes over octonion division algebras}

\author{Elizabeth Callens \and Tom De Medts}

\date{\today}

\maketitle

\vspace*{-5ex}

\begin{abstract}
For every octonion division algebra $\oct$, there exists a projective plane which is parametrized by $\oct$;
these planes are related to rank two forms of linear algebraic groups of absolute type $E_6$.
We study all possible polarities of such octonion planes having absolute points, and their corresponding Moufang set.

It turns out that there are four different types of polarities, giving rise to
(1) Moufang sets of type $F_4$,
(2) Moufang sets of type $^2\!E_6$,
(3) hermitian Moufang sets of type $C_4$, and
(4) projective Moufang sets over a $5$-dimensional subspace of an octonion division algebra.
Case (3) only occurs over fields of characteristic different from two,
whereas case (4) only occurs over fields of characteristic equal to two.
The Moufang sets of type $^2\!E_6$ that we obtain in case (2) are exactly those corresponding to linear algebraic groups of type $^2\!E^{29}_{6,1}$;
the explicit description of those Moufang sets was not yet known.
\end{abstract}



\section{Introduction}

Moufang sets were introduced by Jacques Tits in \cite{Durham} as an axiomization of the isotropic simple algebraic
groups of relative rank one, and they are, in fact, the buildings corresponding to these algebraic groups,
together with some of the group structure (which comes from the root groups of the algebraic group).
In this way, the Moufang sets are a powerful tool to study these algebraic groups.

In this paper, we study Moufang sets arising from polarities of Moufang planes $\PO$, with $\oct$ an octonion division algebra.
(A {\em polarity} is a duality of order two, i.e.\@ an incidence-preserving but not type-preserving automorphism of $\PO$ whose square is the identity.)
Both F.D. Veldkamp \cite{Unitarygroups}, and more recently N. Knarr and M. Stroppel \cite{knarrstroppel2, knarrstroppel} studied these polarities. 

Veldkamp gave a classification of all polarities with absolute points in characteristic different from two.
He showed that there are only three different types of polarities:
\begin{compactenum}[(1)]
    \item
	the standard polarity;
    \item
	a polarity that only exists if the center $E$ of $\oct$ is a separable quadratic extension of some smaller field $k$; and
    \item
	a polarity arising from an automorphism of order $2$ fixing a sub-quaternion algebra (which relies explicitly on the fact that the characteristic is different from two).
\end{compactenum}
On the other hand, N. Knarr and M. Stroppel give a full classification of all polarities of octonion planes in all characteristics.
It turns out that the first two types described by Veldkamp exist as well in characterstic two, but there is also an additional polarity that only exists when
the characteristic is equal to $2$. 
(The paper \cite{knarrstroppel} also deals with polarities having no absolute points or exactly one absolute point, but these cases do not give rise to Moufang sets.)
At the end of \cite{knarrstroppel}, two open questions appear;
one of them is to determine the centralizer in $\Aut(\mathcal{P}_2(\oct))$ for each of the polarities they describe.
This question is closely related to determining the Moufang sets corresponding to these polarities. 

Our goal is to describe all the Moufang sets that arise from these polarities, thereby answering this question.
More specifically, we obtain these Moufang sets by looking at the natural action of the centralizer of the polarity in the automorphism group of the octonion plane. We find that the two types of polarities that do exist in all characteristics give rise to Moufang sets of type $F_4$ ---arising from the standard polarity, as described by T. De Medts and H. Van Maldeghem \cite{MoufF4}--- and (new) Moufang sets of type $^2\!E_6$, corresponding to forms of algebraic groups of type $^2\!E^{29}_{6,1}$.
The polarity only existing in characteristic different from two corresponds to a class of hermitian Moufang sets,
while the polarity only existing in characteristic two induces projective Moufang sets over a $5$-dimensional subspace of $\oct$.

In this paper, we first give two different descriptions of octonion planes. In the next section we describe the (already known) types of Moufang sets that correspond to these polarities. Next, we give a general procedure how to construct these Moufang sets from polarities of the Moufang plane. 
Finally, we discuss for each of these polarities what their corresponding Moufang set is and construct an isomorphism between already existing types of Moufang sets. We find that one of the polarities results in a new type of Moufang sets, and we show that these Moufang sets are of algebraic origin.
More specifically, we obtain that the centralizer of the polarity is an algebraic group of type~$^2\!E^{29}_{6,1}$.

\section{Octonion planes}

We start by explaining the basic objects that we will use, namely projective planes coordinatized by an octonion division algebra.
We will refer to such planes as {\em octonion planes} or as {\em Moufang planes}.
(The latter terminology is perhaps somewhat ambiguous, since there are of course many other projective planes with the Moufang property,
but it is customary to refer to the non-desarguesian projective planes with the Moufang property as {\em Moufang planes}.)

We describe octonion planes $\PO$, where $\oct$ is an octonion division algebra with center $k$, in two different ways as a point-line incidence geometry $(\mathcal{P},\mathcal{L},*)$.

The first way is the most natural way to describe octonion planes. The point set $\mathcal{P}$ consists of three different types of points. Points of the first type are elements of the form $(a,b)$ with $a,b\in\oct$, points of the second type are $(c)$ with $c\in\oct$ and the last type is only one point which we denote by $(\infty)$.

Similarly, there are three types of lines. The first type consists of the elements $[m,k]$ with $m,k\in\oct$, lines of the second type are elements $[l]$ with $l\in\oct$, the third is the line $[\infty]$.

The incidence relation $*$ between points and lines is as follows:
\begin{alignat*}{2}
	(a,b)&*[m,k]
		& \quad & \iff ma+b=k, \\
	(a,b)&*[l]
		&& \iff a=l,  \\
	(c)&*[m,k]
		&& \iff c=m, \\
	(c)&*[\infty]
		&& \ \text{ for all } c\in\oct,  \\
	(\infty)&*[l] && \ \text{ for all } l\in\oct, \\
	(\infty)&*[\infty].
\end{alignat*}

For the second description of the octonion plane, which we denote by $\hatPO$ to avoid confusion, we define $\oct_3$ as the vector space of $3\times 3$ matrices with entries in $\oct$. The set $\HO$ is then defined as the subspace of $\oct_3$ consisting of elements $x$ of the form
\[
x = \begin{pmatrix}
\alpha_1& -a_3 & \overline{a_2}\\
\overline{a_3}&\alpha_2 & a_1\\
 a_2& -\overline{a_1} & \alpha_3 
\end{pmatrix}
\]
with $\alpha_i\in k$ and $a_i\in\oct$ for all $i\in\{1,2,3\}$.
We will often abbreviate the matrix element $x$ as $(\alpha_1,\alpha_2,\alpha_3; a_1,a_2,a_3)$.
The space $\HO$ is a {\em cubic norm structure}, with norm $N \colon \HO\to k$ given by
\begin{multline*}
	N(\alpha_1,\alpha_2,\alpha_3; a_1,a_2,a_3) \\
	= \alpha_1\alpha_2\alpha_3 + \alpha_1N(a_1)-\alpha_2N(a_2)+\alpha_3N(a_3)-T(a_1a_2a_3),
\end{multline*}
trace $T \colon \HO \times \HO \to k$ given by
\begin{multline*}
	T\bigl((\alpha_1,\alpha_2,\alpha_3; a_1,a_2,a_3), (\beta_1,\beta_2,\beta_3; b_1,b_2,b_3) \bigr) \\
	= \textstyle\sum_{i=1}^3 \alpha_i\beta_i-T(a_1\overline{b_1})+T(a_2\overline{b_2})-T(a_3\overline{b_3}),
\end{multline*}
and with adjoint map $\sharp \colon \HO \to \HO$ given by
\begin{multline*}
	(\alpha_1,\alpha_2,\alpha_3; a_1,a_2,a_3)^\sharp \\
	= \bigl( \alpha_2\alpha_3+N(a_1),\alpha_1\alpha_3-N(a_2), \alpha_1\alpha_2+N(a_3); \\
	-\overline{a_2a_3}-\alpha_1a_1,\overline{a_3a_1}-\alpha_2a_2,-\overline{a_1a_2}-\alpha_3a_3 \bigr),
\end{multline*}
where on the right-hand side $N$ and $T$ are the norm and trace map of $\oct$.
Together with the quadratic maps $U_x \colon \HO\to \HO$ (for each $x$ in $\HO$) given by
\[ U_x(y) = T(x,y)x-x^{\sharp}\times y, \]
the space $\HO$ becomes a {\em quadratic Jordan algebra} over $k$, which we will denote by $J(\HO)$.

We will now describe the projective plane $\hatPO$.
The set of points $\mathcal{\hat{P}}$ consists of elements $(x)$ with $x\in\HO$, $x\neq 0$ and $x^\sharp= 0$.
The line set $\mathcal{\hat{L}}$ consists similarly of elements $[x]$, with $x\in\HO$ different from $0$ and satisfying $x^\sharp = 0$.

We define an incidence relation $\hat{*}$ by
\[
(x)\ \hat{*}\ [y]\Longleftrightarrow T(x,y)=0.
\]
Next, we define a map $\phi$ between the point and line sets of $\PO$ and $\hatPO$:
\begin{align*}
 (a,b)&\mapsto (N(b),-N(a),1;a,\overline{b},b\overline{a})\\
(c)&\mapsto (N(c),-1,0;0,0,-c)\\
(\infty)&\mapsto (1,0,0;0,0,0)\\
[m,k]&\mapsto [-1,N(m),-N(k);-\overline{m}k,\overline{k},m]\\
[l]&\mapsto [0,1,-N(l);-l,0,0]\\
[\infty]&\mapsto [0,0,1;0,0,0]
\end{align*}
One can show that $\phi$ is an isomorphism of projective planes. 
For a more detailed description, we refer to \cite[Section 3]{MoufF4}.

\section{Moufang sets}

In this section, we recall some of the basics of Moufang sets, and we refer to \cite{course} for more details.

A {\em Moufang set} $\mouf = \bigl( X, (U_x)_{x \in X} \bigr)$ is a set $X$ together with a collection of groups
$U_x \leq \Sym(X)$, such that for each $x \in X$:
\begin{compactenum}[(1)]
    \item
	$U_x$ fixes $x$ and acts sharply transitively on $X \setminus \{ x \}$;
    \item
	$U_x^\varphi = U_{x\varphi}$ for all $\varphi \in G := \langle U_z \mid z \in X \rangle$.
\end{compactenum}
The group $G$ is called the {\em little projective group} of the Moufang set.

A typical example is given by the group $G = \PSL(2,k)$ acting on the projective line $X = \mathbb{P}^1(k) = k \cup \{ \infty \}$.

\subsection{An explicit construction of Moufang sets}\label{alternativedefinition}

We will now explain how any Moufang set can be reconstructed from a single root group together with one additional permutation \cite{MoufJor}.

Let $(U, +)$ be a group, with identity $0$, and where the operation $+$ is {\em not necessarily commutative}.
Let $X = U \cup \{ \infty \}$, where $\infty$ is a new symbol.
For each $a\in U$, we define a map $\alpha_a \in \Sym(X)$ by setting
\begin{equation}\label{eq:alpha}
	\alpha_a \colon \begin{cases}
		\infty \mapsto \infty \\
		x \mapsto x+a & \text{ for all $x \in U$}.
	\end{cases}
\end{equation}
Let
\begin{equation*}
U_\infty := \{\alpha_a\mid a\in U\} \,.
\end{equation*}
Now let $\tau$ be a permutation of $U^* := U \setminus \{ 0 \}$.
We extend $\tau$ to an element of $\Sym(X)$ (which we also denote by $\tau$) by setting $0^\tau=\infty$ and $\infty^\tau=0$.
Next we set
\begin{equation}\label{eq:defU}
	U_0 := U_\infty^\tau \text{ and } U_a := U_0^{\alpha_a}
\end{equation}
for all $a\in U$, where $U_x^\varphi$ denotes conjugation inside $\Sym(X)$. Let
\begin{equation}\label{eq:mouf}
	\mouf(U,\tau) := \bigl( X,(U_x)_{x\in X} \bigr)
\end{equation} 
and let
\begin{equation*}
	G := \langle U_\infty,U_0 \rangle = \langle U_x \mid x \in X \rangle \,.
\end{equation*}
Then $\mouf(U, \tau)$ is not always a Moufang set, but every Moufang set can be obtained in this way.
Note that, for a given Moufang set, the map $\tau$ is certainly not unique: different choices for $\tau$ can give rise to the same Moufang set.

\subsection{Projective Moufang sets}\label{projectivemoufangset} 

The simplest class of examples of Moufang sets are the {\em projective Moufang sets} $\PSL_2(D)$, which we now describe;
we refer to \cite[Section 5]{course} for more details.

Let $(D, +, \cdot)$ be an {\em alternative division ring}, i.e.\@ a not necessarily associative ring (with $1$)
such that for each $a\in D^*$ there exists some element $a^{-1}\in D^*$ for which $a\cdot a^{-1}b = b = ba^{-1}\cdot a$ for every $b\in D$.
By the Bruck--Kleinfeld theorem, every alternative division ring is either associative (hence a skew field), or it is an octonion division algebra.

Now let $U = (D,+)$ be the additive group of $D$, and define the following permutation on $U^*$:
\[
\tau \colon U^*\to U^* \colon x\mapsto -x^{-1}.
\]
Then $\mathbb{M}(U,\tau)$ is a Moufang set, which we will denote by $\mathbb{M}(D)$, and which is often referred to as the {\em projective Moufang set over $D$}
or the {\em projective line over $D$}.

This construction can be generalized to arbitrary Jordan division algebras \cite{MoufJor}, but we will not need this here.

\subsection{Hermitian Moufang sets}\label{se:herm}

Let K be a field or a skew-field, let $\sigma$ be an involution of $K$, let $V$ be a right vector space over $K$ and let 
\[
K^{-}_{\sigma}=\{a-a^{\sigma}\mid a\in K\}.
\]
A map $q$ from $V$ to $K$ is a hermitian pseudoquadratic form on $V$ with respect to $\sigma$ if there is a form $h$ on $V$ which is hermitian with respect to $\sigma$ such that $q$ and $h$ satisfy
\begin{compactenum}[(i)]
\item $q(a+b)\equiv q(a)+q(b)+h(a,b) \mod K^{-}_{\sigma}$,
\item $q(at)=t^{\sigma}q(a)t \mod K^{-}_{\sigma}$
\end{compactenum}
for all $a,b\in V$ and $t\in K$. 
We say that $q$ is anisotropic if
\begin{compactenum}[(i)]
\setcounter{enumi}{2}
\item
$q(a)\equiv 0 \mod K^{-}_{\sigma}$ \ only if $a=0$.
\end{compactenum}

Let $(T,\cdot)$ denote the group with underlying set
\[
\{(a,t)\in V\times K\mid q(a)-t\in K^{-}_{\sigma}\}
\]
with $(a,t)\cdot (b,u) := (a+b,t+u+h(b,a))$.
Define a permutation $\tau$ on $T^*$ by setting
\begin{equation}\label{eq:tauherm}
	\tau(a,t)=(at^{-1},t^{-1}).
\end{equation}
Then $\mouf(T,\tau)$ is a Moufang set.
Moufang sets obtained in this way are called {\em hermitian Moufang sets}.
For a more detailed description, we refer to \cite[(11.15), (11.16) and (16.18)]{TitsWeiss}.

\subsection{Moufang sets of type $F_4$}

Let $\oct$ be an octonion division algebra over a commutative field $k$. Let  $x \mapsto \overline{x}$ be the standard involution, $N$ the multiplicative norm with $N(x)=x\cdot \overline{x}$ and $T$ the trace map $T(x)=x+\overline{x}$ on $\oct$. We define a set $U$ with
\[
U:=\{(a,b)\in\oct\times \oct\mid N(a)+T(b)=0\}
\]
and the following (non-abelian) group operation on $U$:
\[
(a,b)+(c,d):= (a+c,b+d-\overline{c}\cdot a)
\]
for all $(a,b),\ (c,d)\in U$.
Define a permutation $\tau$ on $U^*$, by setting
\[
\tau(a,b)=(-a\cdot b^{-1}, b^{-1})
\]
for all $(a,b)\in U^*$.
Then $\mathbb{M}(U,\tau)$ is a Moufang set, which we call a {\em Moufang set of type $F_4$} since it is the Moufang set arising from a linear algebraic group
of relative rank one which is of absolute type $F_4$. We refer to \cite{MoufF4} for more details.


\section{The polarities of $\PO$ and $\hatPO$}

In this paragraph, we investigate all polarities of $\hatPO$ (and thus of $\PO$ as well) having at least three absolute points.
Our goal is to describe each type of polarity together with the associated Moufang set.
We were inspired by an article of Veldkamp \cite{Unitarygroups}, which considers polarities with absolute points together with their associated groups (over fields of characteristic different from $2$).
We will extend these results to include fields of characteristic $2$, and we will translate some of his result into the more modern language of Moufang sets.
We first give a general approach for constructing Moufang sets from polarities of Moufang planes.
We will use a similar method as in \cite[Section 5]{MoufF4}.

\subsection{Recovering Moufang sets from polarities of the Moufang plane}\label{Reconstruction}

\begin{definition}
	Let $\Delta$ be a projective plane (or more generally, a generalized polygon) with point set $\mathcal{P}$ and line set $\mathcal{L}$.
	A map $\Psi \colon \mathcal{P} \cup \mathcal{L} \to \mathcal{P} \cup \mathcal{L}$ is called a {\em polarity} of $\Delta$ if:
	\begin{compactitem}
	    \item
		$\Psi(\mathcal{P}) = \mathcal{L}$ and $\Psi(\mathcal{L}) = \mathcal{P}$;
	    \item
		$\Psi$ preserves incidence, i.e.\@
		$p * L \iff \Psi(L) * \Psi(p)$ for all $p \in \mathcal{P}, L \in \mathcal{L}$;
	    \item
		$\Psi^2 = 1$.
	\end{compactitem}
	An element $x \in \mathcal{P} \cup \mathcal{L}$ is called {\em absolute} if $x * \Psi(x)$.
	Similarly, a flag $(p,L)$ with $p * L$ is called an {\em absolute flag} if $L = \Psi(p)$ (and consequently also $p = \Psi(L)$).
\end{definition}

Now suppose $\Psi$ is a polarity of $\PO$ having at least three absolute points and $G$ is the {\em little projective group} of $\PO$,
i.e.\@ the subgroup of $\Aut(\PO)$ generated by the root groups (or equivalently, generated by all elations).

First, we determine the set $X$ of absolute points and the subgroup $\Cent_{G}(\Psi)$ of $G$, which is the group of all elements of $G$ that commute with~$\Psi$.
For every element $\sigma\in\Cent_G(\Psi)$ and every $x\in X$, the image $\sigma(x)$ is again an absolute point.
In this way, we obtain a natural action of $\Cent_G(\Psi)$ on the set of all absolute points (or equivalently, of all absolute flags) of the polarity $\Psi$.
As the following result shows, this gives rise to a Moufang set.

This result seems to be well known, and is used in \cite{MVM} for instance,
but we could not find a proof in the literature.
It turns out that the details are somewhat more intricate than one would expect;
we are grateful to Hendrik Van Maldeghem for discussing the details with us.
\begin{proposition}\label{pr:moufpol}
	Let $\Delta$ be a Moufang $n$-gon, and let $G$ be a subgroup of $\Aut(\Delta)$ containing all root groups.
	Let $\Psi$ be a polarity of $\Delta$, and let $X$ be the set of absolute flags of $\Psi$;
	assume that $|X| \geq 3$.
	Let $C = \Cent_G(\Psi)$; then $C$ acts on $X$.
	Let $K$ be the kernel of this action, and let $\overline{C} = C/K$.
	Then:
	\begin{compactenum}[\rm (i)]
	    \item
		For each $x \in X$, let $U_x$ be the intersection of $\overline{C}$ with the unipotent radical
		$U_+ = U_1 \dotsm U_n$ of $\Delta$ with respect to the pair $(\Sigma, x)$, where
		$\Sigma$ is an (arbitrary) apartment of $\Delta$ containing $x$.
		Then $\bigl(X, (U_x)_{x \in X}\bigr)$ is a Moufang set;
		its little projective group is a normal subgroup of $\overline{C}$.
	    \item
		If either $n$ is even, or $n=3$ and each non-absolute line through a given absolute point of $\Delta$
		contains a second absolute point,
		then $K = 1$, and hence the little projective group of $X$ is a normal subgroup of $C = \Cent_G(\Psi)$ itself.
	    \item
		If $\Delta$ is a Moufang polygon arising from a linear algebraic group of relative rank two,
		then $X$ is the Moufang set of a linear algebraic group of relative rank one.
	\end{compactenum}
\end{proposition}
\begin{proof}
Notice that the absolute flags of $\Delta$ with respect to $\Psi$ are two by two opposite, i.e.\@ they lie at maximal distance.
\begin{compactenum}[(i)]
    \item
	Choose two arbitrary elements $x$ and $y$ of $X$, and let $\Sigma$ be an apartment containing the flags $x$ and $y$.
	Label the roots of $\Sigma$ in such a way that the root groups $U_1,\dots,U_n$ are precisely those fixing the flag $x$.
	Then for each $i \in \{ 1,\dots,n \}$, the conjugate $U_i^\Psi$ is precisely $U_{n+1-i}$, so in particular $U_+$ is normalized by $\Psi$.

	We now claim that $V_+ := U_+ \cap \overline{C}$ acts sharply transitively on $X \setminus \{ x \}$.
	To show that it acts transitively, let $z \in X \setminus \{ x \}$ be arbitrary.
	By \cite[(4.11) and (5.3)]{TitsWeiss}, there exists an element $u \in U_+$ mapping $z$ to $y$.
	It follows that $[u, \Psi] = u^{-1} u^\Psi$ is an element of $U_+$ fixing $y$, but then this element fixes $\Sigma$ pointwise.
	This can only be true if $[u, \Psi] = 1$, and hence $u \in V_+$.
	This shows that $V_+$ acts transitively on $X \setminus \{ x \}$.
	Since no non-trivial element of $U_+$ fixes $y$, we conclude that $V_+$ acts sharply transitively, as claimed.

	Similarly, the group $V_- := U_- \cap \overline{C}$ obtained by interchanging the roles of $x$ and $y$,
	acts sharply transitively on $X \setminus \{ y \}$.
	Moreover, note that $U_+$ is normalized by every element of $G$ fixing the flag $x$, and hence $V_+$ is a normal subgroup of $\operatorname{Stab}_C(x)$.
	This shows that $\langle V_+, V_- \rangle$ is the little projective group of a Moufang set with underlying set $X$.
	Finally, observe that every element of $\overline{C}$ conjugates root groups to root groups,
	and hence $\langle V_+, V_- \rangle$ is a normal subgroup of $\overline{C}$.
    \item
	If either $n$ is even, or $n=3$ and each non-absolute line through a given absolute point of $\Delta$
	contains a second absolute point,
	we will show that every element of $G$ fixing all elements of $X$ fixes all elements of~$\Delta$,
	implying that the kernel $K$ of the action is trivial.

	Assume first that $n=3$, and that $(p, p^\Psi)$ is an absolute flag with the property that every line through $p$
	contains a second absolute point (and hence every point on the line $p^\Psi$ is contained in a second absolute line).
	If $g \in G$ fixes all absolute points and all absolute lines, then it has to fix all points on the line $p^\Psi$
	and all lines through the point $p$, and at least one additional point $q$ not on the line $p^\Psi$.
	This can only be true if $g=1$.

	Assume next that $n$ is even.
	Then the absolute flags of $\Psi$ form an ovoid-spread pair.
	(Indeed, observe that every non-absolute flag $(p, p^\Psi)$ induces an absolute flag in the middle of the
	unique minimal path connecting $p$ and $p^\Psi$.)
	A collineation fixing all points of an ovoid is trivial, which proves the claim.
    \item
	In the case that $\Delta$ arises from a linear algebraic group of relative rank two,
	the group $U_+$ is the unipotent radical of a minimal parabolic $k$-subgroup of $G$ with respect to a fixed maximal $k$-split torus $T$;
	the choice of this torus corresponds to the choice of the apartment $\Sigma$.
	The root group $U_\infty$ of the Moufang set $X$ is then obtained by intersecting $U_+$ with the centralizer in $G$ of $\Psi$.
    \qedhere
\end{compactenum}
\end{proof}
\begin{remark}
	The additional condition when $n=3$ in part (ii) of the previous Proposition is necessary,
	as is illustrated by the polarities of type IV, for which all absolute points lie on a
	single line of $\PO$ (see section~\ref{ss:polIV});
	in this case, the action of $C$ on $X$ is not faithful.
\end{remark}

We will now make this result explicit in our situation where $\Delta$ is the projective plane $\PO$, following the approach of \cite[Section 5]{MoufF4}.

We choose an arbitrary element of $X$, and denote it by $\infty$; the corresponding flag of $\PO$ is denoted by $((\infty),[\infty])$.
Next, we choose an apartment $\Sigma$, for instance the apartment through $(0), (0,0)$ and $(\infty)$, through the flag $((\infty),[\infty])$.
The corresponding unipotent radical $U_+$ is the product $U_1 U_2 U_3$ of the three root groups through $(X,((\infty),[\infty]))$.
The first root group $U_1$ is the group of collineations fixing all the points on $[0]$ and the lines through $(\infty)$, the second group $U_2$ fixes all lines through $(\infty)$ and all points on $[\infty]$ while $U_3$ fixes all points on $[\infty]$ and all lines through $(0)$.
We find that these root groups have the following action on $\PO$:
\begin{alignat*}{3}
U_1 &:=\left\{ x_1(M)\mid M\in\oct \right\} &&\text{ where }  x_1(M) \colon &&
\begin{cases}
	(a,b) \mapsto (a,b-Ma), \\
	[m,k] \mapsto [m+M,k],
\end{cases} \\[1ex]
U_2 &:= \left\{ x_2(B)\mid B\in\oct \right\} &&\text{ where } x_2(B) \colon &&
\begin{cases}
	(a,b)\mapsto (a,b+B), \\
	[m,k]\mapsto [m,k+B],
\end{cases} \\[1ex]
U_3 &:= \left\{ x_3(A)\mid A\in\oct \right\} &&\text{ where } x_3(A) \colon &&
\begin{cases}
	(a,b)\mapsto (a+A,b), \\
	[m,k]\mapsto [m,k+mA].
\end{cases}
\end{alignat*}
Therefore, an arbitrary element of $U_+$ is of the form
\[
x(A,B,M) \colon \begin{cases}
(a,b) \mapsto (a+A,b+B-Ma),\\
[m,k] \mapsto [m+M,k+B+mA+MA].
\end{cases} \]
%
The root group $U_\infty$ can now be obtained as the subgroup of $U_+$ consisting of those elements that map $X$ to itself.

Finally, we have to find a permutation $\tau$ of $U_\infty^*$.
Therefore, it suffices to find a collineation $\sigma$ of $\PO$ that commutes with $\Psi$ and interchanges the points $(0,0)$ and $(\infty)$.
Let $\tau$ be the restriction of $\sigma$ to $X$; then $U_{\infty}^{\tau}$ will be a root groop and therefore will coincide with the root group $U_{(0,0)}$ of the Moufang set.
This implies that $\mathbb{M}(U_{\infty},\tau)$ is the Moufang set obtained from $\Psi$ as in Proposition~\ref{pr:moufpol} above.

\subsection{Description of the different types of polarities}\label{se:pol}

We describe all different types of polarities of the Moufang plane with at least three absolute points.
In characteristic different from two this was already done by Veldkamp in the late sixties \cite{Unitarygroups}.
He used the description of the octonion plane we defined as $\hatPO$ and showed that in this case there only exist three types of polarities.
On the other hand, N.~Knarr and M.~Stroppel described all the polarities in characteristic two on $\PO$.
It turns out that also in this case, there are three types of polarities; the first two coincide with two of the polarities found by Veldkamp,
but the third type is a polarity that only exists in the characteristic two case.

\begin{remark}\label{re:std}
All the polarities we present can be seen as the composition of some standard polarity and a collineation on the Moufang plane; see \cite[Theorem 3.4]{knarrstroppel2}.
Furthermore, we may assume (by conjugating the polarity if necessary)
that this collineation is induced by an automorphism of the octonion division algebra $\oct$; see \cite[Theorem 3.6]{knarrstroppel2}.
Such a collineation is easy to write down explicitly, both in $\PO$ and in $\hatPO$:
for $\PO$, it suffices to apply the automorphism on its coordinates;
for $\hatPO$, the collineation is given by applying the automorphism on each of the entries of the matrix for each point and each line.
\end{remark}


By Remark~\ref{re:std}, it suffices to go over the different types of automorphisms on $\oct$ in order to describe all possible polarities with at least three absolute points.

\subsubsection{Polarities of type I -- the standard polarity}


We define a natural polarity $\pi$ on $\hatPO$:
\[
\pi \colon \hatPO\to \hatPO \colon (x)\mapsto [x] .
\]
It is easy to see that this indeed forms a polarity; we refer to \cite[Theorem~4.5]{MoufF4} for a proof that this polarity has enough absolute points. For obvious reasons, we call this polarity the {\em standard polarity}.

We use the isomorphism $\phi$ between $\hatPO$ and $\PO$ to transform the above polarity into the following polarity of $\PO$:
\begin{align*}
(a,b)&\leftrightarrow [-\overline{{a}{b}^{-1}},-\overline{{b}^{-1}}]\\
(a,0)&\leftrightarrow [\overline{{a}^{-1}}]\\
(0,0)&\leftrightarrow [\infty]\\
(c)&\leftrightarrow [\overline{{c}^{-1}}]\\
(0) &\leftrightarrow [0]\\
(\infty) &\leftrightarrow [0,0]
\end{align*}
An easy transformation (see \cite[Section 4.5]{MoufF4} for more details) reduces these polarities to the following more elegant form:
\begin{align*}
(a,b)&\leftrightarrow [\overline{{a}},-\overline{{b}}]\\
(c)&\leftrightarrow [\overline{{c}}]\\
(\infty) &\leftrightarrow [\infty].
\end{align*}

Taking into account the remarks we just made about the form of a general polarity, we find that each polarity of $\PO$ is conjugate to a polarity
of the following form, for some $\eta \in \Aut(\oct)$:
\begin{equation}\label{polarity}
\begin{aligned}
(a,b)&\leftrightarrow [ \eta(\overline{a}), -\eta(\overline{b}) ] \\[.4ex]
(c)&\leftrightarrow [ \eta(\overline{c}) ] \\
(\infty) &\leftrightarrow [\infty].
\end{aligned}
\end{equation}

\subsubsection{Polarities of type II}\label{ss:II}

This type of polarity only exists if the center $E$ of the octonion division algebra $\oct$ is a separable quadratic extension of a subfield $k$
and $\oct$ is obtained by extending scalars from an octonion division algebra over $k$.
So let $\oct_k$ be an octonion division algebra over $k$, and assume that $E/k$ is a separable quadratic extension such that $\oct = \oct_k \otimes_k E$
remains division.
Let $\gamma$ be the non-trivial element of $\operatorname{Gal}(E/k)$; then $\gamma$ is an involution on $E$, which
extends to a non-linear automorphism $\eta$ of $\oct$ by applying the involution to each coefficient with respect to a basis of $\oct_k$.
This automorphism gives rise to a polarity described in equation~\eqref{polarity},
and we will refer to this class of polarities as the {\em polarities of type II}.

\subsubsection{Polarities of type III}\label{ss:III}

The third type of polarity only exists if the characteristic of the center $k$ of $\oct$ is different from two.
(More precisely, when $\ch(k) = 2$, it coincides with the standard polarity.) 
Let $D$ be an arbitrary quaternion subalgebra of $\oct$.
Then $\oct$ decomposes as the direct sum of $D$ and $D^{\perp}$,
and the map
\[ \eta \colon \oct\to \oct \colon  d+d'\mapsto d-d', \]
for all $d\in D$ and $d'\in D^{\perp}$,
is an automorphism of $\oct$.

Again, such an automorphism induces a polarity by equation~\eqref{polarity}.
We will refer to this class of polarities as the {\em polarities of type III}.

\subsubsection{Polarities of type IV}\label{ss:IV}

In contrast with the polarities of type III, we will now describe a type of polarity (or automorphism) that only exists when the characteristic of the field is two.
The reason for this is the following: in characteristic two, all octonion division algebras possess a totally singular subalgebra $D$ of dimension $4$.
If such an algebra exists, one can show it is in fact a subfield of the octonion division algebra; see \cite[Theorem 4.11]{Faulkner}.

We will show explicitly how such an algebra is naturally contained in an octonion division algebra in characteristic two. For this, we rely on the fact that each octonion algebra in characteristic two has a so-called symplectic basis. This is a basis of the form $e$, $a$, $b$, $ab$, $c$, $ac$, $bc$, $(ab)c$ with
$N(a)N(b)N(c)\neq 0$, such that
\[
\langle e,a\rangle =1,\quad \langle b,ab\rangle=N(b),\quad \langle c,ac\rangle =N(c), \quad \langle bc, (ab)c=N(b)N(c)\rangle,
\]
and all other inner products between distinct basis vectors are zero;
see \cite[Section 1.6]{octonions}.
Now let $D$ be the subspace of $\oct$ spanned by the vectors $e$, $b$, $c$ and $bc$;
then $D$ is a totally singular subalgebra, which is a $4$-dimensional subfield of the octonion algebra.

It is easy to see that for each element $z \in\oct\setminus D$, there is a decomposition $\oct=D\oplus Dz$.
Furthermore, we can choose this element $z$ in such a way that $\tr(z)=1$, or equivalently, such that $\overline{z}=z+e$.
The map
\[ \eta \colon\ \oct\to\oct \colon d+d'z \mapsto d+d'\overline{z} \]
for all $d,d'\in D$, is an automorphism of order two on $\oct$.
The {\em polarities of type IV} are now those induced by such an automorphism.

\section{Description of the Moufang sets}

In the previous section we described all different types of polarities with at least three absolute points that can occur on a Moufang plane. Each of these polarities now induces a certain type of Moufang set. In this subsection we give a detailed description of all these Moufang sets. In most of the cases we can identify the Moufang sets we find with some known algebraic Moufang sets. Polarities of type II on the other hand result in Moufang sets that have not yet been described in literature.
We will show that this type of Moufang set is also of algebraic nature; more specifically, these Moufang sets arise from algebraic groups of type $^2\!E^{29}_{6,1}$.

\subsection{General method}\label{se:polgeneral}

For each of the four types of polarities described in section~\ref{se:pol},
we compute the Moufang set corresponding to such a polarity $\Psi$ with the methods discussed in section \ref{Reconstruction}.
Let $\eta \in \Aut(\oct)$ be the automorphism of the octonion algebra $\oct$ corresponding to the type of $\Psi$.
Then the set of absolute points of $\Psi$ is
\[
X=\{ (a,b)\in \oct\times\oct\mid \eta(\overline{a}) \cdot a+\eta(\overline{b})+b=0\}\cup\{\infty\}.
\]
Notice that the flag $(0,0)*[0,0]$ is fixed under the action of $\Psi$, so for an arbitrary element $x(A,B,M)$ of $U_{\infty}$, the flag
\[ x(A,B,M)(0,0)*x(A,B,M)[0,0]=(A,B)*[M,B+MA] \]
has to be fixed under $\Psi$. Since
\[ \Psi((A,B)*[M,B+MA])=[\eta(\overline{A}),\eta(\overline{B})]*\bigl(\eta(\overline{M}),\eta(\overline{{B}+{M}{A}}) \bigr), \]
this implies $A=\eta(\overline{M})$ and $\eta(\overline{B})=B+MA$.
Because $A$ and $B$ alone determine the element $x(A,B,M)$ completely, we will denote this element by $x(A,B)$.

The composition of two arbitrary elements $x(A,B)$ and $x(C,D)$ has the same action on $X$ as the element $x\bigl(A+C,B+D-\eta(\overline{C})A\bigr)$.
This implies that the group $(U,+)$ of the Moufang set we are describing is
\begin{equation}\label{eq:U}
U=\bigl\{ (a,b)\in \oct\times\oct\mid \eta(\overline{a})\cdot a+\eta(\overline{b})+b=0\bigr\},
\end{equation}
with the group operation $+$ on $U$ given by
\[
(a,b)+(c,d)=\bigl(a+c,b+d-\eta(\overline{c})\cdot a \bigr)
\]
for all $(a,b),(c,d)\in U$.

We determine an appropriate permutation ${\tau}$ on $U^*$. Therefore we need a collineation on $\PO$ commuting with $\Psi$.
Inspired by \cite{MoufF4}, we find that the following collineation $\sigma$ has the desired properties:
\begin{equation}\label{eq:colsigma}
{\sigma} \colon
\begin{cases}
(a,b)\mapsto (-{a}{b}^{-1},{b}^{-1}) \\
[m,k]\mapsto [{k}^{-1}{m},{k}^{-1}] .
\end{cases}
\end{equation}
The permutation ${\tau}$ can be defined as the restriction of ${\sigma}$ to $X$, and is thus defined by
\begin{equation}\label{eq:tau}
	{\tau}(a,b)=(-{a}b^{-1},b^{-1})
\end{equation}
for all $(a,b)\in U^*$.

\subsection{Polarities of type I -- Moufang sets of type $F_4$}

The polarities of type I are exactly those described in \cite{MoufF4}.
There, it is shown that the Moufang sets corresponding to the standard polarity are precisely the Moufang sets of type $F_4$ or thus
Moufang sets arising from an algebraic group of type $F^{21}_{4,1}$.

\subsection{Polarities of type II -- Moufang sets of type $^2\!E_6$}\label{Twisted}

Assume now that $\Psi$ is a polarity of type II.
We will prove that in this case, the Moufang set arises indeed from a linear algebraic group of type $^2\!E^{29}_{6,1}$ over $k$.
Therefore we rely on a paper by S.~Garibaldi \cite{Skip} dealing with the algebraic structure of the linear algebraic group in terms of elements of an Albert algebra.
In [{\em loc.\@ cit.\@}, section~2], he gives a description of the linear algebraic groups of type $^2\!E^{29}_{6,1}$ over
fields of characteristic not $2$ or $3$;
the argument given in the proof of \cite[Theorem 4.1]{MoufF4} (which is based on results by J.~Faulkner \cite{Faulkner})
shows that this description remains valid over fields of arbitrary characteristic.

Recall from section~\ref{ss:II} that $\oct = \oct_k \otimes_k E$, where $\oct_k$ is an octonion division algebra over $k$.
Let $J:=J(\HOk)$ be a quadratic Jordan algebra over $k$ with corresponding norm and trace maps $\N$ and $\tr$,
and define the group $M_1(J)$ as the group of isometries on $J$. 
Every element $g$ of $M_1(J)$ has a natural action on $\hatPOk$; more precisely, $g$ induces a collineation $\rho(g)$ on the Moufang plane by defining
\[
\rho(g) \colon 
\begin{cases}
\mathcal{P}\to\mathcal{P} \colon (x)\mapsto (g(x)) \\
\mathcal{L}\to\mathcal{L} \colon [y]\mapsto[g^{\dag}(y)],
\end{cases}
\]
where $g^{\dag}$ is the unique element of $M_1(J)$ such that $\tr(g(x),g^{\dag}(y))=\tr(x,y)$ for all $x,y\in J$.
Moreover, the map $\rho$ is a $k$-isomorphism between $M_1(J)$ and the little projective group $G$ generated by all elations of the Moufang plane.

Furthermore, we define a map $\tau$ as
\[
\tau \colon  J\to J \colon  (\epsilon_1,\epsilon_2,\epsilon_3;\; c_1,c_2,c_3)\mapsto (\epsilon_1,\epsilon_3,\epsilon_2;\; \overline{c_1}, \overline{c_3}, \overline{c_2})
\]
for all $a = (\epsilon_1,\epsilon_2,\epsilon_3;\; c_1,c_2,c_3) \in J$. 
 
We now construct the twisted algebraic group $^{2}\!{M_1}(J)$ corresponding to $M_1(J)$.
This is an algebraic group such that $^{2}\!{M_1}(J)\otimes_k E \cong M_1(J)\otimes_k E \cong M_1(J \otimes E)$,
this last group being the group of all similarities of $J\otimes_k E$ that preserve the extended norm $N\otimes_k E$.

The non-trivial element $\eta$ of the Galois group $\operatorname{Gal}(E/k)$ acts differently on
the groups $M_1(J)$ and $^2\!M_1(J)$.
Indeed, the non-twisted action on an element of $M_1(J)(E):=M_1(J)\otimes_k E$ which we denote by $\iota$, is as follows:
\[
\iota \colon M_1(J)(E)\to M_1(J)(E) \colon  g\otimes \ell\mapsto g\otimes \eta(\ell)
\]
for all $g\in M_1(J), \ell\in E$.
On the other hand, the twisted Galois action on $M_1(J)(E)$ is given by
\[
\iota^* \colon M_1(J)(E)\to M_1(J)(E) \colon  g\otimes \ell \mapsto \tau\circ [\iota({g\otimes \ell})]^{\dag}\circ \tau^{-1},
\]
where $\tau$ and $\dag$ are the natural extensions to $M_1(J)(E)$ of the equally named maps on $M_1(J)$;
see \cite[Section 2]{Skip}.
We will also write $\iota * g$ for $\iota^*(g)$.

We conclude that the corresponding groups over $k$ are given by
\begin{align*}
	M_1(J) &= \{ g\in M_1(J)(E)\mid \iota(g)=g \} \quad \text{ and} \\
	^2\!M_1(J) &= \{ g\in M_1(J)(E)\mid \iota * g =g \}.
\end{align*}
For a more general and detailed description on the process of twisting of linear algebraic groups, we refer for instance to \cite[Kapitel I, \S 3]{Selbach}
or \cite[Section 5.3]{Serre}.

Next, we define the following map $\psi$ on $\hatPO$, the Moufang plane over~$\oct := \oct_k \otimes_k E$:
\[
\psi \colon \mathcal{P}\to \mathcal{L} \colon  (x)\mapsto [(\tau \circ \tilde\eta)(x)],
\]
with $\tilde\eta := {\id_{J}} \otimes \eta$.
It is easy to check this map induces a polarity of the Moufang plane $\hatPO$.

Observe that the isomorphism $\rho$ from $M_1(J)$ to $G$ extends naturally to an isomorphism from $M_1(J)(E)$ to $G(E)$. 
The following theorem shows that under this extended isomorphism, the elements of the twisted group $^2\!M_1(J)$ are mapped to elements of $G(E)$ commuting with the polarity $\psi$.
\begin{theorem}\label{th:2E6}
The map $\rho$ induces a $k$-isomorphism between the group $^2\!M_1(J)$ of elements $g\in M_1(J)(E)$ such that $\iota*g=g$,
and the group $\Cent_{G(E)}(\psi)$ of elements of $G(E)$ that commute with the polarity~$\psi$.
\end{theorem}
\begin{proof}
Since we have already a $k$-isomorphism between $M_1(J)(E)$ and $G(E)$, it suffices to prove that $\psi\circ\rho(g)=\rho(g)\circ\psi$ if and only if $\iota*g=g$.

Write $E = k(\delta)$; an arbitrary element $g \in M_1(J)(E)$ can then be written in the form
\[ g = g_1 \otimes 1 + g_2 \otimes \delta , \]
with $g_1, g_2 \in M_1(J)$.

Let $a=x_1\otimes m_1+\dots + x_s\otimes m_s\in J_E$ with $x_i \in J$, $m_i \in E$ be an element with $a^{\sharp} = 0$;
then
\[
\psi\bigl(\rho(g)\{(a)\}\bigr) = \psi\{(g(a))\} = [\tau\tilde\eta{g(a)}]
\]
while
\[
\rho(g)\bigl(\psi\{(a)\}\bigr) = \rho(g)([\tau\tilde\eta(a)]) = [g^{\dag}\tau\tilde\eta(a)].
\]
We conclude that $\psi \circ \rho(g) = \rho(g) \circ \psi$ if and only if $\tau\tilde\eta g=g^{\dag}\tau\tilde\eta$.

Next, we verify when $\iota*g=g$ or equivalently if $(\iota*g)^{\dag}=g^{\dag}$.
First, we show that $(\iota*g)^{\dag}=\tau[\iota(g)]\tau$.
We compute
\begin{align*}
\tr(\tau x, \tau y)
&= \tr(x,y)\\
&= \tr\bigl((\iota*g)x,(\iota*g)^{\dag}y\bigr)\\
&= \tr\bigl(\tau [\iota(g)]^{\dag}\tau x, (\iota*g)^{\dag}y\bigr)\\
&= \tr\bigl(\iota(g)^{\dag}\tau x,\tau(\iota*g)^{\dag}y\bigr)\\
&= \tr\bigl(\tau x, [\iota(g)]^{-1}\tau (\iota*g)^{\dag} y\bigr),
\end{align*}
for all $x,y\in J_E$. Since the trace $\tr$ is non-degenerate, we can conclude that $\tau = [\iota(g)]^{-1}\tau (\iota*g)^{\dag}$. 

The problem is reduced to proving that $\tau[\iota(g)]\tau =g^{\dag}$ if and only if $\tau\tilde\eta g = g^{\dag}\tau\tilde\eta$,
or equivalently, that $\iota(g) \tilde\eta = \tilde\eta g$.
Notice that
\[ \iota(g) = g_1 \otimes 1 + g_2 \otimes \eta(\alpha), \]
and hence
\begin{align*}
\iota(g) \tilde\eta (x_i \otimes m_i)
&= \iota(g) \bigl( x_i \otimes \eta(m_i) \bigr) \\
&= g_1(x_i) \otimes \eta(m_i) + g_2(x_i) \otimes \eta(\alpha) \eta(m_i) \\
&= \tilde\eta \bigl( g_1(x_i) \otimes m_i + g_2(x_i) \otimes (\alpha m_i) \bigr) \\
&= \tilde\eta \bigl( g(x_i \otimes m_i) \bigr);
\end{align*}
it follows that $\iota(g) \tilde\eta (a) = \tilde\eta g(a)$ for all $a \in J_E$.
\end{proof}

In order to able to invoke Proposition~\ref{pr:moufpol},
we have to check the condition in part (ii) of that proposition.
\begin{lemma}
	Let $\PO$ and $\Psi$ be as above.
	Every non-absolute line $[a]$ through the point $(\infty)$ contains an absolute point of the form $(a,b)$.
\end{lemma}
\begin{proof}
	We have to check that for every $a \in \oct$, the equation
	\[ \eta(\overline{a}) \cdot a+\eta(\overline{b})+b=0 \]
	has a solution $b \in \oct$.
	Since $\eta(\overline{a}) \cdot a$ is fixed by the map $x \mapsto \eta(\overline{x})$,
	this follows immediately from \cite[Proposition 3.2]{knarrstroppel}.
\end{proof}

It now follows from Proposition~\ref{pr:moufpol} that the Moufang set corresponding to $\Psi$ has a little projective group
which is a normal subgroup of the centralizer $\Cent_G(\Psi)$ of the polarity.
If we can show that this polarity $\Psi$ gives rise to the same Moufang set (up to isomorphism) as the Moufang set we obtained from the polarity $\psi$, then it will
indeed follow from Theorem~\ref{th:2E6}
that the Moufang set corresponding to $\Psi$ arises from a twisted algebraic group of type $^2\!E^{29}_{6,1}$.

We calculate what this polarity $\psi$ looks like on $\PO$, and we find
\begin{align*}
(a,b)& \mapsto [-{\eta(\bar{b})^{-1}},-\eta(\overline{a b^{-1}})]\\
(a,0) & \mapsto [-\eta(\overline{a})]\\
(0,0) & \mapsto [0]\\
(c) & \mapsto [0,-\eta(\bar{c})^{-1}]\\
(0) & \mapsto [\infty]\\
(\infty) & \mapsto [0,0]
\end{align*}
for all $a\in\oct$ and all $b,c\in\oct\setminus \{0\}$. 
All that is left now is to construct an incidence preserving coordinate transformation mapping $\Psi$ to $\psi$.
After some calculations, we find that the following transformation $T \colon \PO \to \PO$ does the job:
\[
\begin{aligned}
	(a,b)&\mapsto (b,-a)\\
	 (c)&\mapsto (-c^{-1})\\
	(0)&\mapsto (\infty)\\
	(\infty)&\mapsto (0)
\end{aligned}
\hspace*{8ex}
\begin{aligned}{}
	[m,k]&\mapsto [-m^{-1},-m^{-1} k]\\
	[0,b]&\mapsto [b]\\
	[a]&\mapsto [0,-a]\\
	[\infty]&\mapsto [\infty]
\end{aligned}
\]
for all $a,b,k\in\oct$ and all $c,m\in\oct\setminus \{0\}$.
The existence of such a transformation proves that both Moufang sets are indeed isomorphic.
We obtain the following result:
\begin{theorem}
The Moufang set $\mathbb{M}(U,\tau)$ obtained from a polarity of type~II, given by equations~\eqref{eq:U} and~\eqref{eq:tau}, is the Moufang building associated to a twisted linear algebraic group of type $^2\!E^{29}_{6,1}$.
Conversely, every Moufang set corresponding to such an algebraic group of type $^2\!E^{29}_{6,1}$ can be obtained from a polarity of type II
and is therefore of the form $\mathbb{M}(U,\tau)$, with $U$ and $\tau$ as in equations~\eqref{eq:U} and~\eqref{eq:tau},
and with $\eta$ as in section~\textup{\ref{ss:II}}.
\end{theorem}

\subsection{Polarities of type III -- Moufang sets of hermitian type}

We reconstruct the Moufang structure on $X$ arising from of a polarity $\Psi$ of type III, and we will prove that the Moufang set we obtain is indeed a hermitian Moufang set (of type $C_4$),
by constructing an explicit isomorphism.

First, we determine the structure of the Moufang set obtained by the polarity of the octonion plane.
We begin in the same fashion as in section~\ref{se:polgeneral};
in particular, the group $U$ is given by equation~\eqref{eq:U}.
However, in order to simplify our proof of the isomorphism with the hermitian Moufang set,
we will choose a slightly different $\tau$, as follows.
Instead of the collineation $\sigma$ given by equation~\eqref{eq:colsigma}, we choose the following collineation $\sigma$
(which still commutes with the polarity $\Psi$):
\[
{\sigma} \colon
\begin{cases}
(a,b)\mapsto \bigl(-\eta(ab^{-1}), \eta(b)^{-1}\bigr)\\
[m,k]\mapsto \bigl[\eta(k^{-1}m), \eta(k)^{-1}\bigr].
\end{cases}
\]
(Recall that $\eta \in \Aut(\oct)$ is as in section~\ref{ss:III}.)
The corresponding permutation $\tau$ on $U^{*}$ is thus defined by
\begin{equation}\label{eq:tauIII}
	{\tau}(a,b)=\bigl(-\eta(ab^{-1}),\eta(b)^{-1}\bigr)
\end{equation}
for all $(a,b)\in U^*$.

\begin{theorem}
	Let $k$ be a field with $\ch(k) \neq 2$, and let $\oct$ be an octonion division algebra over $k$.
	Consider a decomposition $\oct = D\oplus cD$ with $D$ a quaternion subalgebra of $\oct$ and some $c\in {D^{\perp}}$ with $\beta = N(c)\neq 0$.
	Let
	\[ \eta \colon \oct \to \oct \colon a_1 + c a_2 \mapsto a_1 - c a_2 , \]
	and let
	\[ h \colon \oct \times \oct \to D \colon (a_1 + c a_2, b_1 + c b_2) \mapsto \overline{a_1} b_1 + \beta \overline{a_2} b_2 \]
	for all $a_1,a_2,b_1,b_2 \in D$.
	Then:
	\begin{compactenum}[\rm (i)]
	    \item
		$h$ is a hermitian form on $\oct$ (considered as a $2$-dimensional right vector space over $D$);
	    \item
		the Moufang set corresponding to $h$ (as defined in section~\textup{\ref{se:herm}})
		is isomorphic to the Moufang set $\mouf(U, \tau)$ arising from the polarity $\Psi$ corresponding to $\eta$,
		with $U$ given by equation~\eqref{eq:U} and $\tau$ given by equation~\eqref{eq:tauIII};
	    \item
		$\mouf(U, \tau)$ arises from a linear algebraic group of type $C_4$.
	\end{compactenum}
\end{theorem}
\begin{proof}
\begin{compactenum}[\rm (i)]
    \item
	This is obvious from the definitions.
	Note that the corresponding hermitian pseudoquadratic form $q$ is equal to $\tfrac{1}{2} N_\oct$:
	\[ q(a_1 + ca_2) = \tfrac{1}{2} \bigl(\N(a_1) + \beta \N(a_2) \bigr) . \]
    \item
	We will denote an element $a_1 + c a_2 \in \oct$ as $(a_1, a_2)$, where $a_1,a_2 \in D$.
	Observe that the multiplication in $\oct$ with respect to this decomposition is given by
	\[ (a_1, a_2) (b_1, b_2) = \bigl( a_1 b_1 - \beta b_2 \overline{a_2}, \, b_1 a_2 + \overline{a_1} b_2 \bigr) . \]
	Now let $a=(a_1,a_2)$ and $b=(b_1,b_2)$;
	then the condition
	\[ \eta(\overline{a})\cdot a+\eta(\overline{b}) + b = 0 \]
	occuring in equation~\eqref{eq:U} can be rewritten as the system of the following two equations:
	\begin{align*}
	& \overline{a_1}\cdot a_1-\beta \overline{a_2}\cdot a_2 +\overline{b_1}+b_1=0\quad \text{and}\\
	& a_1 a_2+b_2=0.
	\end{align*}
	Now consider the group $T = \{(a,t)\in \oct \times D \mid q(a) - t \in D^{-}_{\sigma}\}$
	as defined in section~\ref{se:herm}, with group operation $(a,t) \cdot (b,u) := (a+b,t+u+h(b,a))$.
	Since $\ch(k) \neq 2$, the space $D^-_\sigma$ is precisely the subspace of trace zero elements of $D$, and hence
	\[ T = \bigl\{ (a_1, a_2, t) \in D \times D \times D \mid \N(a_1) + \beta \N(a_2) = \T(t) \bigr\} , \]
	with
	\[ (a_1, a_2, t) \cdot (b_1,b_2, u) = \bigl( a_1 + b_1, a_2 + b_2, t + u + \overline{b_1} a_1 + \beta \overline{b_2} a_2 \bigr) . \]
	It turns out that the map
	\[ \chi \colon T \to U \colon (a_1, a_2, t) \mapsto \bigl( (a_1, \overline{a_2}), (-t + \beta \N(a_2), -a_1 \overline{a_2}) \bigr) \]
	is a group isomorphism.

	It remains to check that the map $\tau$ given by equation~\eqref{eq:tauIII} corresponds to the map $\tau$ given by equation~\eqref{eq:tauherm}
	under the isomorphism $\chi$.
	This follows from another short calculation, keeping in mind that
	\[ \N\bigl( -t + \beta \N(a_2) \bigr) + \beta \N\bigl( -a_1 \overline{a_2} \bigr) = N(t) \]
	for all $(a_1, a_2, t) \in T$.
	We leave the details to the reader.
    \item
	By (ii), this now follows, for instance, from \cite[p. 56]{Boulder}.
    \qedhere
\end{compactenum}
\end{proof}

\subsection{Polarities of type IV}\label{ss:polIV}

We finally assume that $\Psi$ is a polarity of type IV; in particular $\ch(k) = 2$.
Again, $U$ and $\tau$ are given by equations~\eqref{eq:U} and~\eqref{eq:tau}, respectively.
In this case, however, the group $U$ takes a very simple form.
Indeed, it follows from \cite[Theorem 7.3]{knarrstroppel} that
\[
U = \bigl\{(0,y)\mid y\in\oct,\, \eta(\overline{y})=y\bigr\},
\]
and hence $U$ becomes an abelian group, with
\[ (0, y_1) + (0, y_2) = (0, y_1 + y_2) \]
for all $(0, y_1), (0, y_2) \in U$.
The map $\tau$ is now simply given by
\[ \tau(0,y) = (0, y^{-1}) \]
for all $(0,y) \in U$.

Suppose $D$ is the $4$-dimensional subfield of $\oct$ fixed by $\eta$.
Then \cite[Theorem~3.1]{knarrstroppel} tells us that the set of elements fixed under
$x \mapsto \eta(\overline{x})$ is the set $\Fix(\eta)\oplus kz = D \oplus kz$, with $z$ as in section \ref{ss:IV}.
Observe that the map $\tau$ does indeed preserve the subset $U^*$.

We conclude that the Moufang set associated to $\Psi$ is (isomorphic to) a Moufang subset of the projective Moufang set $\mathbb{M}(\oct)$
over the octonion division algebra $\oct$, as defined in section~\ref{projectivemoufangset};
the root groups are $5$-dimensional subspaces of the $8$-dimensional vector space $\oct$ over $k$.



\bigskip\hrule\bigskip
\small

\noindent
Elizabeth Callens \\
Ghent University, Dept.\@ of Mathematics \\
Krijgslaan 281 (S22), B-9000 Gent, Belgium \\
{\tt elcallen@cage.UGent.be}

\medskip

\noindent
Tom De Medts \\
Ghent University, Dept.\@ of Mathematics \\
Krijgslaan 281 (S22), B-9000 Gent, Belgium \\
{\tt tdemedts@cage.UGent.be}

\end{document}